\documentclass[a4paper]{article}

\usepackage{amsmath, amsfonts, amsthm}
\usepackage{hyperref}
\usepackage[alphabetic,initials]{amsrefs}
\usepackage{graphicx}
\usepackage{tikz-cd}
\usepackage{mathrsfs}
\usepackage[utf8]{inputenc}

\newtheorem{theorem}{Theorem}[section]
\newtheorem{lemma}[theorem]{Lemma}
\newtheorem{corollary}[theorem]{Corollary}
\newtheorem{proposition}[theorem]{Proposition}
\newtheorem*{theoremIntro}{Main Theorem}
\newtheorem*{corollaryIntro}{Corollary}

\theoremstyle{definition}
\newtheorem{definition}[theorem]{Definition}
\newtheorem{remark}[theorem]{Remark}

\newcommand{\co}{\colon\thinspace} 
\newcommand{\RR}{\mathbb{R}} 
\newcommand{\CC}{\mathbb{C}} 
\newcommand{\ZZ}{\mathbb{Z}} 
\newcommand{\TT}{\mathbb{T}} 
\renewcommand{\SS}{\mathbb{S}} 
\DeclareMathOperator{\Id}{Id} 
\let\P\relax
\DeclareMathOperator{\P}{\mathcal{P}} 
\DeclareMathOperator{\Po}{\mathcal{P}_{\!\!o}} 
\newcommand{\D}{\mathcal{D}} 
\newcommand{\Do}{{\mathcal{D}_{\!o}}} 
\DeclareMathOperator{\MCG}{\pi_0\D} 
\newcommand{\F}{\mathscr{F}} 
\newcommand{\bdot}{\mathbin{\vcenter{\hbox{\scalebox{0.4}{$\bullet$}}}}}
\newcommand{\inclusion}{\hookrightarrow}

\def\CS{{\mathcal{CS}}}
\def\from{\co}
\def\R{{\RR}}
\def\Z{{\mathbb{Z}}}
\def\cU{{\mathcal{U}}}
\def\cV{{\mathcal{V}}}
\def\rst#1{|_#1}
\def\eps{\varepsilon}
\def\del{\partial}

\DeclareMathOperator{\SL}{SL} 

\title{On the contact mapping class group\\ of Legendrian circle bundles}

\author{Emmanuel Giroux \and Patrick Massot}

\AtEndDocument{\bigskip{\footnotesize%
\noindent
E.~Giroux: \textsc{Unité de Mathématiques Pures et Appliquées, École Normale
Supérieure de Lyon, CNRS} \par  
\noindent
\texttt{emmanuel.giroux@ens-lyon.fr} \par
\addvspace{\medskipamount}
\noindent
P.~Massot: \textsc{Centre de Mathématiques Laurent Schwartz, École polytechnique}\par  
\noindent
\texttt{patrick.massot@polytechnique.edu} \par
}}

\begin{document}

\maketitle

\begin{abstract}
In this paper, we determine the group of contact transformations modulo contact
isotopies for Legendrian circle bundles over closed surfaces of nonpositive
Euler characteristic. These results extend and correct those presented by the
first author in \cite{Giroux_transfo}. The main ingredient we use is 
connectedness of certain spaces of embeddings of surfaces into contact
3-manifolds. In the third section, this connectedness question is studied in
more details with a number of (hopefully instructive) examples.
\end{abstract}

In this paper, we study contact transformations of $3$-manifolds which are
circle bundles equipped with contact structures tangent to the fibers. The main
example of such a manifold is the unit cotangent bundle $V = T_1^*S$ of a
surface $S$, endowed with its canonical contact structure $\xi$: this contact
manifold is also called the manifold of cooriented contact elements over $S$.
Other examples are obtained as follows: for any positive integer $d$ dividing
$|2g-2|$ where $g$ is the genus of $S$, the manifold $V$ admits a $d$-fold
fibered cyclic cover $V_d$ and the pullback $\xi_d$ of $\xi$ on $V_d$ is a
contact structure tangent to the fibers of $V_d$ over $S$. It is  a nice and
easy observation that all Legendrian circle bundles are of this form (see
\cite[p.\ 179]{Lutz_pivot}).

Our goal here is to determine the contact mapping class group of $(V_d, \xi_d)$,
namely the group $\pi_0 \D (V_d ; \xi_d)$ where $\D (V_d ; \xi_d)$ denotes the 
group of contact transformations of $(V_d, \xi_d)$. This group has an
obvious homomorphism to the usual (smooth) mapping class group 
$\pi_0 \D (V_d)$ (where $\D (V_d)$ consists of all diffeomorphisms of
$V_d$) which has been computed in \cite{Waldhausen_klasse}. By standard
fibration results (see Section 1), the kernel of this homomorphism is
tightly related to the fundamental group of the isotopy class of
$\xi_d$, \emph{i.e.} the connected component of $\xi_d$ in the space
$\CS(V_d)$ of all contact structures on $V_d$. Note here that a contact
transformation (or contactomorphism) is a diffeomorphism preserving the
contact structure and its coorientation. 

Our main result is the following theorem, in which $V_d$ is endowed with any
principal circle bundle structure inherited from one on $V = T_1^*S$---the case
$g=1=d$ was previously treated by H.~Geiges and M.~Klukas in \cite{GeigesT3}:
 
\begin{theoremIntro}[Theorem \ref{thm:diffeos_tangents} and Corollary
\ref{cor:pi1}]
Let $S$ be a closed, connected, orientable surface of genus $g \ge 1$ and $d$ a
positive integer dividing $|2g-2|$. Denote by $R_t \from V_d \to V_d$ the action
of $2\pi t \in \R/2\pi\Z$ by rotation along the fibers. Then:
\begin{itemize}
\item
The fundamental group $\pi_1 \bigl(\CS(V_d), \xi_d\bigr)$ is infinite cyclic and
generated by the loop $(R_t)_*\xi_d$, $t \in [0,1/d]$.
\item
The kernel of the natural homomorphism
$$ \pi_0 \D (V_d; \xi_d) \to \pi_0 \D (V_d) $$
is the cyclic group $\Z/d\Z$ spanned by the contact mapping classes of the deck
transformations of $V_d$ over $T_1^*S$. 
\end{itemize}
\end{theoremIntro}

As a direct consequence, we obtain:

\begin{corollaryIntro}[Corollary~\ref{cor:elements_contact}]
Let $S$ be a closed orientable surface of genus $g \ge 2$. Then the natural
homomorphism
$$ \pi_0 \D (S) \to \pi_0 \D (T_1^*S ; \xi) $$
induced by the differential is an isomorphism.
\end{corollaryIntro}

This corollary is stated as Theorem 1 in \cite{Giroux_transfo} but the ``proof''
given there contains a mistake. See Section 3.1 for a detailed erratum and
Section 3.2 for several related examples.

In the case $g=1$, each manifold $V_d$ is diffeomorphic to $\TT^3 = (\RR/\ZZ)^3$
fibering over $S = \TT^2$ by the projection $(x,y,z) \mapsto (x,y)$, and its
contact structure $\xi_d$ can be defined by
\[ \cos(2d\pi z)\, dx - \sin(2d\pi z)\, dy = 0, \quad
   x, y, z \in  \RR/\ZZ. \]
Then:

\begin{theoremIntro}[Theorem~\ref{thm:t3}]
The image of the obvious homomorphism
$$ \pi_0 \D (\TT^3, \xi_d) \to \pi_0 \D (\TT^3) = \SL_3(\ZZ) $$
is the subgroup of transformations preserving $\ZZ^2 \times \{0\} \subset
\ZZ^3$. In particular, this subgroup is isomorphic to $\pi_0 \D(\TT^3,\xi_1)$.
\end{theoremIntro}

Finally, for $g=0$, an unpublished result of M.~Fraser \cite{Fraser_P2} 
shows that the contact transformation group of the standard projective
$3$-space (namely, the unit cotangent bundle of the $2$-sphere) is connected.
This completes the list of contact mapping class groups for unit cotangent
bundles of closed orientable surfaces.

\section{Natural fibrations in contact topology} \label{sec:fibrations}

For any compact manifold $V$ with (possibly empty) boundary, we denote
by $\D (V, \partial V)$ the group of diffeomorphisms of $V$ relative to
a neighborhood of the boundary. When the boundary of $V$ is empty, we
sometimes drop $\partial V$ from our notations.

\begin{lemma} \label{lemma:Gray}
Let $(V, \xi)$ be a compact contact manifold. The natural map
$$ \D (V, \partial V) \to \D (V, \partial V) \bdot \xi, \quad
   \phi \mapsto \phi_*\xi, $$
is a locally trivial fibration whose fiber is the contact transformation group
$\D (V, \partial V; \xi) \subset \D (V, \partial V)$. 
\end{lemma}

\begin{proof}
By the classical Cerf-Palais fibration criterion (see \cite[Lemma 2 p. 240, \S 
0.4.40]{Cerf_these} or \cite[Theorem~A]{Palais}), it suffices to show that the
above map admits a continuous local section near every point $\xi_0 \in \D (V,
\partial V) \bdot \xi$. Choose a vector field $\nu$ transverse to $\xi_0$ and
observe that $\xi_0$ has a convex open neighborhood $\cU$ in $\D (V, \partial V)
\bdot \xi$ which consists of contact structures transverse to $\nu$. Then Gray's
theorem associates to any contact structure $\xi_1 \in \cU$ an isotopy $\phi_t
\in \D (V, \partial V)$, $t \in [0,1]$, such that $\phi_0 = \Id$ and $\phi_{t*} 
\xi_0 = (1-t) \xi_0 + t \xi_1$ for all $t \in [0,1]$. Moreover, one can easily
arrange that this isotopy varies continuously with $\xi_1$. Indeed, it is
uniquely determined by a smooth path $\sigma_t$ of sections $V \to TV / \xi_t$
and, if this path depends continuously on $\xi_1$ (if $\sigma_t$ is identically 
$0$, for instance), then so does the corresponding isotopy, and the map $\xi_1
\mapsto \phi_1$ gives the required continuous section. 
\end{proof}

Assume from now on that the contact manifold $(V, \xi)$ has dimension $3$  and
let $F$ be a compact orientable surface properly embedded in $V$. Recall that
the characteristic foliation $\xi F$ of $F$ in $(V, \xi)$ is the (singular)
foliation spanned by the line field $\xi \cap TF$ (the singularities are the
points where $\xi = TF$). We denote by
\begin{itemize}
\item
$\P (F, V)$ the space of proper embeddings $F \to V$ which coincide with the
inclusion $\iota \from F \to V$ near $\partial F$;
\item
$\Po (F, V) \subset \P (F, V)$ the connected component of the inclusion $\iota$;
\item
$\P (F,V;\xi) \subset \P (F,V)$ the subspace of embeddings $\psi$ which induce
the same characteristic foliation as the inclusion, \emph{i.e.} satisfy 
$\xi\, \psi(F) = \psi_* (\xi F)$;
\item
$\Po (F, V; \xi)$ the intersection $\Po (F, V) \cap \P (F, V; \xi)$.
\end{itemize}
The same standard tools as in the proof of Lemma \ref{lemma:Gray} give the
following:

\begin{lemma} \label{lemma:fibration_PFxi}
Let $(V,\xi)$ be a compact contact manifold of dimension $3$. For every properly
embedded surface $F \subset V$, the restriction map 
$$ \D (V, \partial V; \xi) \to \P (F, V; \xi), \quad
   \phi \mapsto \phi \rst F, $$ 
is a locally trivial fibration over its image.
\end{lemma}

\begin{proof}
Each embedding $\psi_0 \in \P (F, V; \xi)$ which lies in the image is also in
the image of the restriction map
$$ \D (V, \partial V) \to \P (F, V), \quad
   \phi \mapsto \phi \rst F, $$ 
which is a locally trivial fibration by Cerf-Palais's fibration theorem. As a
result, there exists a neighborhood $\cV$ of $\psi_0$ in $\P (F, V; \xi)$ and a
continuous extension map $\cV \to \D (V, \partial V)$ which associates to every 
embedding $\psi_1 \in \cV$ a diffeomorphism $\phi_1 \in \D (V, \partial V)$ such
that $\psi_1 = \phi_1 \rst F$. Using Gray's theorem and the fact that embeddings
in $\cV$ induce the same characteristic foliation, it is easy to correct this
extension map to that it takes values in $\D (V, \partial V; \xi)$. We conclude
applying again the Cerf-Palais fibration criterion.
\end{proof}

\begin{remark} \label{rem:rel_neighb}
The above lemma is a typical result where it is useful to work relatively to a 
neighborhood of the boundary and not just to the boundary itself. Indeed, any
diffeomorphism relative to both $\partial V$ and a properly embedded surface $F$
is tangent to the identity along $\partial F$, and so the fibration property
fails in this case. However, since the inclusion of $\D (V, \partial V; \xi)$
into the group of contact transformations relative to the boundary is a homotopy
equivalence, this does not matter.
\end{remark}

We now recall how the theory of $\xi$-convex surfaces can be used to study the
homotopy type of $\P (F, V; \xi)$ (see \cites{Giroux_thesis, Giroux_transfo}).
Let $F$ be a compact orientable surface properly embedded in $(V, \xi)$ with
(possibly empty) Legendrian boundary; $F$ is \emph{$\xi$-convex} if it admits a 
\emph{homogeneous neighborhood}, namely a product neighborhood
$$ U := F \times \R \supset F = F \times \{0\}, \quad \text{with} \quad
   \partial U = \partial F \times \R \subset \partial V, $$
in which the vector field $\partial_t$, $t \in \R$, preserves $\xi$. The points
$p \in F$ where $\partial_t(p) \in \xi$ then form a multi-curve $\Gamma$ called 
the \emph{dividing set} of $F$ associated with $U$. This curve depends on $U$ and
its product structure but its isotopy class does not and is uniquely determined 
by the foliation $\xi F$: specifically, $\Gamma$ is the unique multi-curve (up
to isotopy) which avoids the singularities of $\xi F$, is transverse to $\xi F$ 
and divides $F$ into regions where the dynamics of $\xi F$ is alternatively
expanding and contracting (see \cites{Giroux_thesis, Giroux_transfo} for more
details). It follows that the curves dividing a given singular foliation on a
surface in the above sense form a contractible space. Among them, the dividing
sets associated with all possible homogeneous neighborhoods $U$ of $F$ are those 
intersecting $\partial F$ at the points where $\xi$ is tangent to $\partial V$. 
Moreover, we have the following proposition (see \cite[Lemmas 6 and 7]
{Giroux_transfo}):

\begin{proposition} \label{prop:divided_foliations}
Let $F$ be a $\xi$-convex surface, $U$ a homogeneous neighborhood, and $\Gamma$ 
the associated dividing set.

\item[\textrm{(a)}] 
The space $\F(F;\Gamma)$ of singular foliations on $F$ which are tangent to
$\partial F$ and admit $\Gamma_U$ as a dividing set is an open contractible
neighborhood of $\xi F$ in the space of all singular foliations on $F$.

\item[\textrm{(b)}]
There exists a continuous map $\F (F;\Gamma) \to \P (F,V)$, $\sigma \mapsto 
\psi_\sigma$, with the follwing properties:
\begin{enumerate}
\item
$\psi_{\xi F}$ is the inclusion $F \to V$;
\item
$\psi_\sigma(F)$ is contained in $U = F \times \R$  and transverse to the
contact vector field $\partial_t$ for all $\sigma \in \F (F;\Gamma)$;
\item
$\xi\, \psi_\sigma(F) = \psi_\sigma(\sigma)$ for all $\sigma \in \F (F;\Gamma)$.
\end{enumerate}

\item[\textrm{(c)}]
Let $\P(F,V;\Gamma)$ denote the space of embeddings $\psi \in \P(F,V)$ such that
$\psi(F)$ is $\xi$-convex with dividing set $\psi(\Gamma)$. Then the inclusion
$\P(F,V;\xi) \to \P(F,V;\Gamma)$ is a homotopy equivalence.
\end{proposition}

We will also need the following result which shows that the homotopy type of
$\D (V, \partial V; \xi)$ is locally constant when $\partial V$ is $\xi$-convex
(see \cite[Proposition 8]{Giroux_transfo}:  

\begin{proposition} \label{prop:homotopy_depends_Gamma}
Let $V$ be a compact $3$-manifold, $\Delta$ a multi-curve on $\partial V$ and
$\CS(V,\Delta)$ the space of contact structures $\xi$ on $V$ for which 
$\partial V$ is $\xi$-convex with dividing set $\Delta$. For $\xi \in
\CS(V,\Delta)$, the homotopy type of $\D (V, \partial V; \xi)$ depends only on
the connected component of $\CS(V;\Delta)$ containing $\xi$.
\end{proposition}

\section{Legendrian circle bundles over surfaces}
\label{sec:main}

\subsection{The general case}

In this section we consider a compact oriented surface $S$ which is neither a
sphere nor a torus. The torus case will be treated in Section~\ref{ss:torus}.
Actually, using results of \cite{Patrick_gcs}, the following discussion can be
carried over to orbifolds. 

As in the introduction, $(V_d, \xi_d)$ denotes the $d$-fold fibered cyclic cover
of the unit cotangent bundle $V := V_1 = T_1^*S$, equipped with the pullback of
the canonical contact strucutre $\xi$ of $V$.

In any trivialization of $\pi : V_d \to S$ over a subsurface $R \subset S$ with
non-empty boundary, the contact structure $\xi_d$ can be described as follows:
if $J$ denotes the complex structure on $S$ induced by our choice of a principal
bundle structure on $T_1^*S$ then the restricition of $\xi_d$ to $V_d \rst R
\simeq R \times \SS^1$ has the form
\[ \xi_{\lambda, d}
 = \ker \Big(\cos (d\theta) \lambda + \sin (d\theta)\lambda \circ J \Big), \quad
   \theta \in \SS^1, \]
where $\lambda$ is a non-singular $1$-form on $R$. In practice, we will take $R$
equal to $S$ if $S$ has non-empty boundary and to $S$ with an open disk removed
if $S$ is closed. 

Now observe that the preimage $F := \pi^{-1}(\gamma)$ of any properly embedded
curve $\gamma$ in $S$ is a $\xi_d$-convex surface in $V_d$. Indeed, any vector
field $X$ in $S$ transverse to $\gamma$ (and tangent to $\del S$) lifts to a
contact vector field $\bar X$ transverse to $F$ (and tangent to $\del V_d$). The
dividing set of $\xi_d F$ associated with $\bar X$ is the set of points in $F$
where $\xi_d$ projects down (by the differential of $\pi$) to the line spanned
by $X$.

If $\gamma$ is contained in the subsurface $R$, the trivialization of $V_d \rst
R$ induces a diffeomorphism $F = \pi^{-1}(\gamma) \simeq \gamma \times \SS^1$.
Then the dividing set, provided all its components are consistently oriented (as
parallel curves), represents the homology class $2 (d, x-1) \in \ZZ^2 = H_1 
(\gamma \times \SS^1)$, where $x$ denotes the index of $\lambda$ along $\gamma$.

The following proposition can be proved using \cite[Lemma~4.7]{Giroux_bif} 
(a special case of the semi-local Bennequin inequality proved
later as \cite[Proposition~4.10]{Giroux_cercles}) exactly as in 
\cite[Lemma~3.9]{Giroux_cercles} which dealt with circle bundles without
boundary.

\begin{proposition} \label{prop:decoupage}
Let $F$ be a torus fibered over a homotopically essential circle in $S$ and
$\Gamma$ a dividing set for $\xi_d F$. For any isotopy $\varphi$ such that
$\varphi_1(F)$ is also $\xi_d$-convex, the foliation $\xi_d \varphi_1(F)$ is
divided by a collection of curves isotopic to the components of
$\varphi_1(\Gamma)$.
\end{proposition}

We now turn to spaces of embeddings of surfaces. The following lemma is useful
to prove the existence of contact transformations which are smoothly but not
contact isotopic to the identity:

\begin{lemma} \label{lemma:embeddings_torus}
Let $T$ be a fibered torus over a homotopically non-trivial circle $C$ in $S$
and $i \co T \to V_d$ the inclusion map. Let $R_t$ be the action of $e^{2i\pi 
t}$ on $V_d$. For any non-zero integer $k$ in $\ZZ$, the path 
$\gamma_k \co [0,1] \to \Po (T,V_d)$ defined by $\gamma_k(t) = R_{kt/d} \circ i$
is non-trivial in $\pi_1 (\Po (T,V_d), \Po(T,V_d;\xi_d))$.

For any integer $k$ between $1$ and $d-1$, the action of $R_{k/d}$ on
$\pi_0 (\Po (T,V_d;\xi_d))$ is non-trivial.
\end{lemma}

The above lemma will be reduced to the following statement.

\begin{proposition}[{\cite[Proposition~7.1]{Paolo_T3}}] \label{prop:LegT2R}
In $\big(\TT^2 \times \RR, \ker \big(\cos(2n\pi z)dx - \sin(2n\pi z)dy\big)
\big)$, the Legendrian circles 
$\{0\} \times \SS^1 \times \{0\}$ and $\{0\} \times \SS^1 \times \{k\}$ are not
contact isotopic for any $k \ne 0$.
\end{proposition}

Alternatively, one could use the stronger result due to Eliashberg, Hofer and
Salamon \cite{Eliashberg_Hofer_Salamon} saying that, 
in $\big(\TT^3, \ker\big( \cos(2n\pi z)dx - \sin(2n\pi z)dy\big)\big)$, the
Legendrian circle $\{0\} \times \SS^1 \times \{0\}$ cannot be displaced from the
pre-Lagrangian torus $\SS^1 \times \SS^1 \times \{0\}$ by a contact isotopy.
However this result uses holomorphic curves in symplectizations so it has a
different flavor
from the techniques we use in this paper.

\begin{proof}[Proof of Lemma~\ref{lemma:embeddings_torus}]
We first explain how to reduce the first statement to a statement in a thickened
torus.
Suppose for contradiction that there is an path $j$ in $\Po(T, V_d; \xi_d)$
from the inclusion to $\gamma_k(1)$. The path lifting property of the fibration
$\D(V; \xi_d) \to \P(T, V_d; \xi_d)$ of Corollary \ref{lemma:fibration_PFxi}
gives a contact isotopy $\varphi_t$ such that 
$\gamma_k(t) = \varphi_t \circ i$. 
Let $p \co \hat S \to S$ be the covering map associated to the subgroup
generated by $C$ in $\pi_1(S)$ and let $\hat V$ be the induced circle bundle
over $\hat S$.  We denote by $\hat T$ the compact component of $p^{-1}(T)$. The
contact isotopy $\varphi$ lifts to a contact isotopy $\hat \varphi$ for the
induced contact structure $p^*\xi_d$. The interior of $\hat S$ is an open
annulus and contains a closed sub-annulus $A$ such that $\hat\varphi_t(\hat T)$
stays above $A$ for all $t$.  One can then cut-off $\hat\varphi_t$ using
Libermann's theorem to get a contact isotopy with support in $A \times \SS^1$.
In $H_1(\hat T, \ZZ)$, we consider a $\ZZ$-basis $(S, F)$ where $F$ is the
homology class of fibers. The torus $\hat T$ has circles of singularities which
can be oriented consistently to get a total homology class $2(d S + m F)$ where
$m$ is an unknown integer.  After pulling back everything under a $d$-fold
covering map from $A$ to itself, we can assume there are exactly $2d$ circles
of singularities.  The circle bundle over $A$ then embeds into $\TT^3$ equipped
with $\ker\big( \cos(2d\pi z)dx - \sin(2d\pi z)dy\big)$ so it is sufficient to
get a contradiction there.  Note that we don't make any claim concerning how
the circle bundle structure embeds inside $\TT^3$ and we won't use it.

We now consider the covering map $\TT^2 \times \RR \to \TT^3$ sending 
$(x, y, s)$ to $(x, y, s \mod \ZZ)$. The contact isotopy $\hat\varphi$ lifts to
a contact isotopy contradicting Proposition~\ref{prop:LegT2R}.

The statement about $\pi_0(\Po(T, V_d; \xi_d))$ follows immediately from what
we proved and the long exact sequence of the pair 
$(\Po(T, V_d), \Po(T, V_d; \xi_d))$ since the path corresponding to $k$ between
$1$ and $d - 1$ does not come from a loop in $\Po(T, V_d)$.
\end{proof}

\begin{proposition}
\label{prop:tores_Paolo_Patrick}
If $T$ is a fibered torus over a non-separating embedded circle in $S$ then 
the group of deck transformations of $V_d \to V$ acts freely and 
transitively on $\pi_0(\Po(T, V_d; \xi_d))$. If $A$ is a fibered annulus over a
non-separating properly embedded arc in $S$ then $\Po(A, V_d; \xi_d)$ is connected.
\end{proposition}

The following proof will need one more technical ingredient from the study of
contact structures on circle bundles: the twisting number.
Fibers in $V$ have a canonical framing coming from vector fields along
the fiber which project to some constant vector on the base. If $L$ is any
Legendrian curve isotopic to a fiber, we call twisting number of $L$ the number
$t(L)$ of turns made by $\xi_d$ along $L$ compared to the canonical framing
transported by isotopy from some fiber, see \cite[Page 227]{Giroux_cercles} for
further discussion.  As explained in \cite[Lemma 3.6]{Giroux_cercles}, it
follows from Bennequin's inequality in $\RR^3$ that $t(L) \leq -d$ for all $L$.

\begin{proof}
We first prove connectedness of $\Po(A, V_d; \xi_d)$. Let $(j_t)_{t \in [0, 1]}$ be an
isotopy of embeddings of $A$ which coincides with the inclusion map on a
neighborhood of $\partial A$. Let $\Gamma$ be a dividing set on $A$ associated
to some homogeneous neighborhood.
According to Proposition~\ref{prop:divided_foliations}.c, we only need to prove
that $j$ is homotopic to a path in $\Po(A, V_d; \Gamma)$. 

We use Colin's discretization technique \cite{Colin_spheres} which relies on
the following observation.  We can find times $t_0 = 0 < t_1 < \cdots < t_k =
1$ such that for $t$ in $[t_i, t_{i + 1}]$ the annuli $j_t(A)$ are all
contained in some pinched product: $A_i \times [0, 1]$ with $\{x\} \times [0,
1]$ collapsed to a point for $x$ in a neighborhood of $\partial A_i$. The
sub-path $j_t$, $t \in [t_i, t_{i + 1}]$ is then homotopic with fixed end
points to the concatenation of two paths whose common extremity has image $A_i
\times \{0\}$. In addition, genericity of $\xi_d$-convex surfaces allows us to
assume $A_i \times \{0\}$ is $\xi_d$-convex.  So we have replaced $j$ by the
concatenation of $2k$ paths of embeddings sweeping out pinched products. We can
assume there is only one such path and the general case follows by induction.

The classification of tight contact structures on solid tori \cites{Giroux_bif,
Honda_I} guaranties that, in this situation, $j$ is homotopic to an isotopy in
$\Po(A, V_d; \Gamma)$ as soon as $A'=j_1(A)$ is divided by a curve $\Gamma'$ isotopic
to $j_1(\Gamma)$. So we prove that fact.

We first remark that both $\Gamma$ and $\Gamma'$ are made of $2d$ traversing
curves because otherwise we could use the realization lemma
(Proposition~\ref{prop:divided_foliations}.b) to produce a Legendrian circle
$L$ isotopic to the fibers with twisting number $t(L) > -d$. We now need two
separate arguments depending on the value of $d$.

Suppose first that $d$ is greater than one. Let $A''$ and $A'''$ be annuli
isotopic to $A'$ through $\xi_d$-convex surfaces and such that the annuli, $A$,
$A'$, $A''$, $A'''$ pairwise bound pinched products and there is an arc going
from $A$ to $A'''$ in the pinched product they bound and meeting $A'$ and then
$A''$ in its interior. Near their boundary all these annuli are fibered over
arcs in $S$.  Let $T \times [0, 1]$ be a thickened torus with $\xi_d$-convex
boundary $T_0 \sqcup T_1$ such that $T_0$ is a smoothing of $A \cup A'''$,
$T_1$ is a smoothing of $A' \cup A''$ and those tori are fibered in the
smoothing region.  We identify the first homology groups of $T_0$ and $T_1$
using the product $T \times [0, 1]$ and fix an integer basis $(S, F)$ where $F$
is the class coming from fibers of $V$.

If $\Gamma'$ is not isotopic to $j_1(\Gamma)$ then, after orienting all dividing
curves of $T_0$ and $T_1$ in the same way, their total homology class is 
$2(d, x_0)$ on $T_0$ and $2(d, x_1)$ on $T_1$ with $x_1 \neq x_0$. Pick's
formula ensures that the triangle with vertices $(0,0)$, $(d,x_0)$ and 
$(d, x_1)$ in $H_1(T; \RR)$ contains integer points outside its vertical edge.
If $(a, b)$ is such a point then $a < d$. The classification of tight contact
structures on thickened tori then gives a $\xi_d$-convex torus in 
$T \times [0, 1]$ divided by a collection of parallel curves which can be
oriented all in the same way to have total homology class $2(a, b)$. The
realization lemma gives again a Legendrian curve with twisting number 
$t = -a > -d$ hence a contradiction.

We now handle the case $d = 1$. In particular $(V_d, \xi_d)$ is isomorphic to
$(V, \xi)$. Adding 1-handles, one can easily embed $S$ into
a surface $S'$ with connected boundary so that $\xi$ extends to a contact
structures tangent the fibers of $S' \times \SS^1$ and, of course, the
projection of $A$ stays non-separating in $S'$. So we assume $S$ has connected
boundary. Let $S''$ be the surface obtained by gluing a disk along the
boundary of $S$.
The description of the contact structure $\xi$ in terms of a 1-form 
$\lambda$ on $S$ allows to understand the characteristic foliation 
$\xi \partial V$ in term of the index of $\lambda$ along $\partial S$. The later
is given by the Poincar\'e-Hopf theorem so it is fixed. Using this information
we can embed $(V, \xi)$ into the space $V''$ of contact elements of $S''$,
with its canonical contact structure. 
In $V''$ one can extend $A$ and $A'$ to isotopic non-separating tori which
coincide outside $V$. Both tori are divided by two curves and Proposition
\ref{prop:decoupage} guaranties those curves are isotopic. This implies that
$\Gamma'$ is isotopic to $j_1(\Gamma)$.

We now turn to the case of a torus $T$ fibered over a non-separating circle $C$
in $S$. We fix a dividing set $\Gamma$ of $\xi_d T$.
Let $j$ be any isotopy of embeddings of $T$ in $V_d$ such that $T' = j_1(T)$ is
$\xi_d$-convex. Proposition \ref{prop:decoupage} ensures that any component of any
dividing set of $\xi_d T'$ is isotopic in $T'$ to a component of $j_1(\Gamma)$.
However there always exist isotopies of $T$ which change the number of dividing
curves and, as explained in Section~\ref{sec:disconnected}, there is no general
result allowing to get rid of them. This is where we need $C$ to be
non-separating and not only homotopically non-trivial. The idea, which was born
in \cite[Proposition 5.4]{Paolo_thesis} and further developed in \cite[Lemma
8.5]{Patrick_gcs} is to consider a fibered torus $F$ intersecting $T$ along one
fiber. Then for any isotopy $\varphi_t$, one can discretize the movement of $F$
while constructing an isotopy of $T$ through $\xi_d$-convex surfaces. There is no
boundary in \cite{Patrick_gcs} but one can check that it does not change
anything here. This trick constructs a contact isotopy $\varphi'_t$ such that
$\varphi'_1(T) = \varphi_1(T)$. Of course parametrizations don't match in
general: $(\varphi'_1)^{-1} \circ \varphi_1$ induces a self-diffeomorphism of
$T$ which may fail to be isotopic to the identity among diffeomorphisms
preserving $\xi_d T$. However, after composing $\varphi'_1$ by a deck
transformation, we can assume that each circle of singularities of $T$ is
globally preserved and after an ultimate contact isotopy, we get a path in
$\Po(T, V_d; \xi_d)$.  So the deck transformations group acts transitively on
$\pi_0(\Po(T, V_d; \xi_d))$. The last part of Lemma~\ref{lemma:embeddings_torus}
states that this action is also free. 
\end{proof}

\begin{theorem} \label{thm:diffeos_tangents}
If $S$ is closed then the kernel of the canonical homomorphism 
$\MCG(V_d,\xi_d) \to \MCG(V_d)$ is the cyclic group of deck transformations of
$V_d$ over $V$. 
If $V$ has non-empty boundary then 
$\MCG(V_d, \partial V; \xi_d) \to \MCG(V_d, \partial V_d)$ is injective.
\end{theorem}

\begin{proof}
We first assume $V$ has non-empty boundary and prove that the map
$\MCG(V_d, \partial V_d; \xi_d) \to \MCG(V_d, \partial V_d)$ is
trivial. The proof proceeds by induction on 
\[
n(S) = -2\chi(S) -\beta(S) = \beta(S) + 4g(S) -4
\]
where $\chi(S)$ and $g(S)$ are the Euler characteristic and genus of $S$ and
$\beta(S)$ is the number of connected components of $\partial S$.  So 
$n(S) \geq -3$ with equality when $S$ is a disk. 

We first explain the induction step so we assume $n(S) > -3$. Let $\varphi$ be a
contactomorphism of $V_d$ relative to some neighborhood $U$ of $\partial V_d$
and smoothly isotopic to the identity relative to $U$. Let $a$ be a properly
embedded non-separating arc in $S$ and
denote by $A$ the annulus fibered over $a$ and $i \co A \to V_d$ the inclusion
map. According to Proposition~\ref{prop:tores_Paolo_Patrick}, 
$\Po(A, V_d,\xi_d)$ is connected hence the path lifting property of the fibration
$\Do(V_d, \partial V_d; \xi_d) \to \Po(A, V_d; \xi_d)$ from
Lemma~\ref{lemma:fibration_PFxi} implies that $\varphi$ is contact isotopic to 
some $\varphi'$ which is relative to $A$ and $U$. Using
Remark~\ref{rem:rel_neighb}, we can assume $\varphi'$ is relative to a
neighborhood of $\partial V_d \cup A$ which is fibered over some neighborhood $W$
of $a \cup \partial S$ in $S$. We cut $S$ along $a$ and round the corners
inside $W$ to get a subsurface $S' \subset S$ with $n(S') < n(S)$. By induction
hypothesis applied to $\pi^{-1}(S')$, $\varphi'$ is contact isotopic to
identity so the induction step is completed.

The induction starts with the disk case which is already explained with all
details in \cite[Page 345]{Giroux_transfo}. The idea is the same as for the
induction step but the cutting surface in the solid torus $V_d$ is a meridian
disk. There are no such disk with Legendrian boundary in $V_d$ but one can use
the realization lemma (Proposition \ref{prop:divided_foliations}) to deform
$\xi_d$ near $\partial V_d$ until such a disk exists. This does not change the
homotopy type of $\D(V_d, \partial V_d; \xi_d)$ according to Proposition
\ref{prop:homotopy_depends_Gamma}. A variation on Colin's result about
embedding of disks in \cite[Theorem 3.1]{Colin_recollement} then replaces
Proposition~\ref{prop:tores_Paolo_Patrick} and the final isotopy is provided by
Eliashberg's result in \cite{Eliashberg_20ans} that 
$\MCG(B^3, \partial B^3; \xi)$ is trivial for the standard ball.

We now turn to the case where $V_d$ is closed. We first prove that the group of
deck transformations injects into $\MCG(V_d; \xi_d)$. Let $C$ be a non-separating
circle in $S$ and $T$ the fibered torus over $C$. Denote by $i$ the inclusion of
$T$ in $V_d$. Proposition~\ref{prop:tores_Paolo_Patrick} guaranties that the
action of a non-trivial deck transformation $f$ on $\pi_0(\Po(T, V_d; \xi_d))$
is non-trivial hence $f$ is non-trivial in $\MCG(V_d; \xi_d)$.

We now prove surjectivity. Let $\varphi$ be a contactomorphism of $V_d$ which is
smoothly isotopic to the identity. Proposition~\ref{prop:tores_Paolo_Patrick}
gives a deck transformation $f$ such that $f \circ \varphi \circ i$ is isotopic
to $i$ in $\Po(T, V_d; \xi_d)$. As above, this implies that $f \circ \varphi$ is
contact isotopic to a contactomorphism $\varphi'$ which is relative to an open
fibered neighborhood $U$ of $T$. The circle bundle $V_d \setminus U$ has
non-empty boundary hence we know that $\varphi'$ is contact isotopic to
identity.
\end{proof}

\begin{corollary}
\label{cor:pi1}
Assume that $V_d$ has empty boundary and denote by $\CS$ the space
$\D(V_d)\bdot\xi_d$ of contact structures isomorphic to $\xi_d$ on $V_d$.
Let $R_t$ denote the action of $e^{2i\pi t} \in \SS^1$ on $V_d$. 
The fundamental group $\pi_1(\CS, \xi_d)$ is an infinite
cyclic group generated by the loop $t \mapsto (R_{t/d})_*\xi_d$.
\end{corollary}

\begin{proof}
The fibration $\D(V_d) \to \CS$ of Lemma~\ref{lemma:Gray} gives the exact
sequence
\[
\pi_1(\D(V_d), \Id) \to \pi_1(\CS, \xi_d) \to \pi_0(\D(V_d; \xi_d)) \to
\pi_0(\D(V_d)).
\]
We know from \cites{Laudenbach_dim3, Hatcher_large} that $\pi_1(\D(V_d), \Id)$ is an
infinite cyclic group generated by the loop $t \mapsto R_t$, $t \in [0, 1]$.
Lemma~\ref{lemma:embeddings_torus} implies that this group injects into
$\pi_0(\Do(V_d), \Do(V_d; \xi_d)) \simeq \pi_1(\CS, \xi_d)$.
The result then follows from Theorem~\ref{thm:diffeos_tangents} describing 
the kernel of $\pi_0(\D(V_d; \xi_d)) \to \pi_0(\D(V_d))$.
\end{proof}

Before coming back to the special case of $V = V_1$ we note one more general
corollary of Proposition~\ref{prop:decoupage}.

\begin{lemma}
\label{lemma:over_identity}
A diffeomorphism of $V_d$ which is fibered over the identity is isotopic to a
contactomorphism only if it is isotopic to the identity.
\end{lemma}

\begin{proof}
Let $f$ be a diffeomorphism of $V_d$ fibered over the identity of $S$. In order to
guaranty that $f$ is isotopic to the identity, it is enough to check that,
for every torus $T$ fibered over a homotopically essential circle, the
restriction of $f$ to $T$ preserves an isotopy class of curves which is
different from the class of fibers. 
Assume that $f$ is isotopic to a contactomorphism. This condition means that
$\xi_d' = f^*\xi_d$ is isotopic to $\xi_d$. 
Since $f$ is fibered, the contact structure $\xi_d'$ is
also tangent to fibers. Proposition~\ref{prop:decoupage} then implies that, for
each torus $T$ as above, $f$ preserves the isotopy class of dividing curves.
Those dividing curves are homotopically essential and not isotopic to fibers
hence $f$ is isotopic to the identity.
\end{proof}

\begin{corollary}
\label{cor:elements_contact}	
The lifting map from $\MCG(S)$ to $\MCG(V, \xi)$ is an isomorphism.
\end{corollary}

\begin{proof}
We denote by $p$ the projection from $V$ to $S$ and by $\D(S, \partial S)$
the group of diffeomorphisms of $S$ relative to a neighborhood of $\partial S$.
In the sequence of maps:
\[
\pi_0\D(S, \partial S) \to \pi_0\D(V, \partial V; \xi) \to \pi_0\D(V, \partial V)
\]
the composite map is known to be injective (this follows from considerations of
fundamental groups) so the first map is also injective.
It remains to prove that it is surjective.  Let $\varphi$ be a
contactomorphism. We want to prove that $\varphi$ is contact isotopic to the
lift of some diffeomorphism of $S$.  According to Waldhausen \cite[Satz
10.1]{Waldhausen_klasse}, $\varphi$ is smoothly isotopic to a fibered
diffeomorphism $f$: there exists an isotopy $\psi$ and a diffeomorphism 
$\bar f$ in $\D(S, \partial S)$ such that $f = \psi_1 \circ \varphi$ and 
$p \circ f = \bar f \circ p$.  We will prove that $\varphi$ is contact isotopic 
to the lift $D\bar f$. We first note that $f \circ D(\bar f)^{-1}$ is fibered 
over the identity and is smoothly isotopic to a contactomorphism (through
the path $t \mapsto \psi_{1-t} \circ \varphi \circ D(\bar f)^{-1}$).
Lemma~\ref{lemma:over_identity} then guaranties that $f \circ D(\bar f)^{-1}$
is smoothly isotopic to the identity. Hence $\varphi$ is smoothly isotopic to
$D\bar f$ hence contact isotopic according to
Theorem~\ref{thm:diffeos_tangents}.
\end{proof}

\subsection{The torus case}
\label{ss:torus}

We now explain how the previous discussion can be modified to handle the case
of a torus base. On $\TT^3 = (\RR/\ZZ)^3$ with coordinates $(x, y, z)$, we set 
\[
\xi_d = \ker\big( \cos(2d\pi z)dx - \sin(2d\pi z)dy\big).
\]
The case $d = 1$ corresponds to the contact element bundle of $\TT^2$ while
higher values of $d$ come from self-covering maps unwrapping the fibers.
We denote by $R_t$ the map $(x, y, z) \mapsto (x, y, z + 2\pi t)$.

\begin{theorem}
	\label{thm:t3}
On $\big(\TT^3, \xi_d\big)$:
\begin{itemize}
\item 
a diffeomorphism is isotopic to a contactomorphism if and only its action on
$H_2(\TT^3)$ preserves the homology classes of the prelagrangian torus 
$\{z = 0\}$ up to sign.
\item
The kernel of $\MCG(\TT^3; \xi) \to \MCG(\TT^3)$ is isomorphic to the cyclic
group of order $d - 1$ generated by $(x, y, z) \mapsto (x, y, z + 1/d)$.
\item
The fundamental group $\pi_1(\CS, \xi_d)$ is an infinite cyclic group generated
by the loop $t \mapsto R_{t/n}^*\xi$, $t \in [0, 1]$.
\end{itemize}
\end{theorem}

The first point comes directly from the classification of isotopy classes of
tight contact structures on $\TT^3$ that we now recall. 
The first result, proved in \cite{Giroux_tordue}, is that all incompressible
prelagrangian tori in $(\TT^3, \xi_d)$ are isotopic to $\{z = 0\}$. In
particular they share a common homology class which is well defined up to sign
in $H_2(\TT^3)$.
Next recall that the torsion of a contact manifold $(V, \xi)$ was defined in
\cite[Definition 1.2]{Giroux_bif} to be the supremum of all integers $n \ge 1$
such that there exist a contact embedding of  
\[
\left(T^2 \times [0, 1], \ker\left(\cos(2d\pi z)dx - 
\sin(2d\pi z)dy\right)\right), \qquad (x,y,z) \in T^2 \times [0,1]
\] 
into the interior of $(V, \xi)$ or zero if no such integer $n$ exists.
It follows from \cite[Proposition~3.42]{Giroux_bif} that the torsion of $\xi_d$ 
on $\TT^3$ is $d - 1$ .
The classification of isotopy classes of tight contact structure on $\TT^3$
established in \cite{Giroux_bif} is that any tight contact structure is
isomorphic to some $\xi_d$ and two of them are isotopic if and only if they
have the same torsion and their incompressible prelagrangian tori are
homologous. The first point of the above theorem follows from this
classification and the obvious observation that isomorphic contact structures
have the same torsion.

The description of the kernel in the second point has exactly the same proof as
in the preceding section.

The third point is slightly different because $\pi_1(\D(\TT^3), \Id)$ has rank three.
It is generated by the three obvious circle actions on $(\SS^1)^3$. 
However, two of these circle actions actually belong to $\D(\TT^3; \xi_d)$ so that, in
the exact sequence
\[
\pi_1(\D(\TT^3), \Id) \to \pi_1(\CS, \xi_d) \to \pi_0(\D(\TT^3; \xi_d)) \to
\pi_0(\D(\TT^3)).
\]
considered in the proof of Corollary \ref{cor:pi1}, the extra generators 
of $\pi_1(\D(\TT^3, \Id))$ are mapped to trivial elements of $\pi_1(\CS, \xi_d)$
and the end result does not change.

\section{Examples of disconnected spaces of embeddings}
\label{sec:disconnected}

\subsection{Erratum about the reference [Gir01b]}

In the proof of \cite[Proposition 10]{Giroux_transfo}, the assertion following
the words \emph{Grâce au lemme 14} (namely, the claim that \emph{il est possible
de trouver des points $s_0 = 0 < s_1 < \dots < s_k = 1$ tels que, pour $0 \le i 
\le k-1$, les surfaces $F_s$, $s \in [s_i,s_{i+1}]$ soient toutes incluses dans 
un voisinage rétractile $U_i$ de $F_{s_i}$}) is wrong. In fact, the statement of
Proposition 10 turns out to be wrong (Proposition~\ref{prop:isotopiesT3A} below
provides a counter-example). As a consequence, Lemma 19 and the proofs 
of Theorems 1, 3 and 4 are also wrong. Other intermediate results (in particular
Lemma 15 which is sometimes useful as a complement to contact convexity theory)
are presumably correct. 

To be more explicit, Lemma 14 allows us to assume that each surface $F_s$ has a
retractible neighborhood $U_s$. Each neighborhood $U_s$ contains all surfaces
sufficiently close to $F_s$; in other words, there exists a positive number
$\eps_s$ such that $F_t \subset U_s$ as soon as $|s-t| < \eps_s$. The intervals 
$J_s = (s-\eps_s, s+\eps)$ form an open covering of the segment $[0,1]$ so this
covering admits a finite subcovering $J_{s_i}$, $0 \le i \le k$, and the points
$s_i$ can be chosen and indexed so that $s_0 = 0$, $s_k = 1$ and $s_{i-1} < s_i$
for $1 \le i \le k$. Thus, each $U_i = U_{s_i}$ is a retractible neighborhood of
$F_{s_i}$ such that $U_i \supset F_s$ for all $s$ with $|s-s_i| < \eps_{s_i}$.
This does not imply that $U_i \supset F_s$ for all $s \in [s_i,s_{i+1}]$ and, in
general, it is impossible to force $U_i$ to contain $F_{s_{i+1}}$. Here is an
example: assume that $F_s$ is a convex torus for $s \notin \{1/3,2/3\}$ and that
$s = 1/3$ (resp. $s = 2/3$) corresponds to the death (resp. the birth) of a pair
of parallel components in the dividing set; then $F_{1/3}$ cannot appear in any 
retractible neighborhood of any $F_s$ with $s < 1/3$ because of
\cite[Lemma~4.7]{Giroux_bif} (in contrast, $F_{2/3}$ can
appear in a retractible neighborhood of some $F_s$ with $1/3 < s < 2/3$,
see \cite[Lemma~3.31b]{Giroux_bif}).

\subsection{Disconnected spaces of embeddings}

In this section we describe examples of disconnected spaces consisting of
smoothly isotopic embeddings inducing a fixed characteristic foliation. 
Those examples should be compared to the connectedness results which were
crucial in Section~\ref{sec:main} and complement the erratum above.  More
specifically, we construct disconnected spaces of smoothly isotopic $\xi$-convex
embeddings with a fixed dividing set, and the former spaces are deformation
retracts of the latter by Proposition~\ref{prop:divided_foliations}.

\begin{proposition}
In $\SS^3$ equipped with its standard contact structure, let $D$ be an
unknotted immersed disk with a single clasp self intersection and such that the
contact structure is tangent to $D$ along its boundary (here unknotted means
that $D$ has a regular neighborhood which is an unknotted solid torus, see
Figures~\ref{fig:DW} and \ref{fig:immersedOT}). Let $W$ be any unknotted solid
torus which is a regular neighborhood of $D$. Assume that $T = \partial W$ is
$\xi$-convex. The space $\Po(T, \SS^3; \xi)$ is not connected.
\end{proposition}

\begin{figure}[ht]
	\centering
	\includegraphics{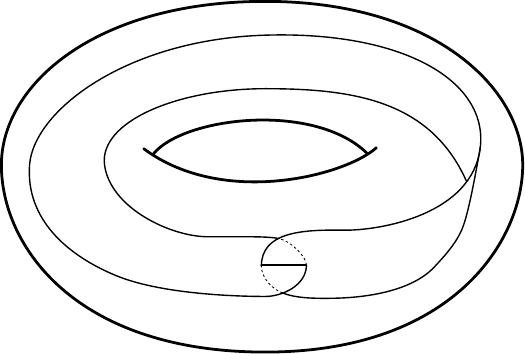}
	\caption{An unknotted immersed disk with a single clasp, sitting inside an
	unknotted solid torus}
	\label{fig:DW}
\end{figure}

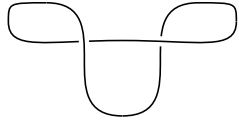
\begin{figure}[ht]
	\centering
\begin{tikzpicture}
  \draw[draw=white, double=black, ultra thick] (0,-1) .. controls +(1,0) and +(-1,0) .. (1,.5) ..
  	controls +(.5,0) and +(0,.25) .. (1.5,.25);
  \draw[draw=white, double=black, ultra thick] (1.5, .25) .. controls +(0,-.4) and +(1,0) .. (0,0)
  	.. controls +(-1,0) and +(0,-.4) .. (-1.5,.25) 
  	.. controls +(0,.25) and +(-.5, 0) .. (-1,.5);
  \draw[draw=white, double=black, ultra thick] (-1,.5)
  	.. controls +(1, 0) and +(-1,0) .. (0 ,-1) ; 
\end{tikzpicture}
	\caption{Lagrangian projection of the boundary of an immersed overtwisted
	disk in the standard contact $\RR^3$}
	\label{fig:immersedOT}
\end{figure}

\begin{proof}
We coorient $T$ so that $W$ is on the negative side of $T$ and we denote
by $W'$ the solid torus which is the closure of $\SS^3 \setminus W$. Since
$\xi$ orients $\SS^3$, we also get an orientation on $T$. This orientation
induces a cyclic ordering on $P(H_1(T; \RR))$. We set 
$d = \ker(H_1(T) \to H_1(W))$ and $d' = \ker(H_1(T) \to H_1(W'))$ where maps
are induced by inclusion. A direction in $H_1(T)$ distinct from $d$ and $d'$
will be called positive if it lies between $d$ and $d'$ and negative otherwise.

If $T_1$ is any cooriented unknotted torus then we can repeat the above
discussion and, for any isotopy sending $T$ to $T_1$ (preserving
orientations), positive directions will get identified with positive
directions because such isotopies map meridian disks to meridian disks in each
solid torus.

\noindent\textbf{Claim:} Any unknotted $\xi$-convex torus in $\SS^3$ is divided
by a collection of essential closed curves whose direction is positive.

\begin{proof}[Proof of claim]
Because $\xi$ is tight, we know from \cite[Théorème 4.5]{Giroux_cercles} that
$T$ is divided by a collection of parallel homotopically essential circles. 
We fix a positive basis $(\mu, \mu')$ of $H_1(T)$ such that $\mu$ is in $d$ and
$\mu'$ is in $d'$ (using the notations above). Let $\Gamma$ be a dividing set
for $T$. For some choice of orientation, the components of $\Gamma$ have
homology class $p\mu + q\mu'$ with $q \ge 0$. The realization lemma (recalled
as part of Proposition~\ref{prop:divided_foliations}) allows to perturb $T$ so
that the characteristic foliation $\xi T$ has a circle of singularities
parallel to $\Gamma$. Such a circle $L$ is a Legendrian $(p, q)$ torus knot
along which $\xi$ does not twist compared to $T$. The Seifert framing of $L$
differs from the framing coming from $T$ by $pq$ so that the Thurston-Bennequin
invariant of $L$ is $-pq$. Since the genus of $L$ is $(|p|-1)(q-1)/2$, the
Bennequin inequality gives $-pq \leq |p|q - |p| - q$. This condition is
equivalent to $p > 0$ and $(q-1)(p-1) \ge 1/4$ and, since $p$ and $q$ are
integer, it is equivalent to $p \ge 1$ and $q \ge 1$.
\end{proof}

We can see $\SS^3$ as the union of two unknotted curves transverse to $\xi$ 
and an open interval of pre-lagrangian tori whose directions sweep out all
positive directions. For each rational positive direction $d$ and each
positive integer $n$, we can perturb the corresponding pre-Lagrangian torus 
to a $\xi$-convex torus $T'$ divided by $2n$ curves with direction $d$. So one
of those $T'$ has the same dividing set as $T$ up to isotopy. After using
once more the realization lemma, we can ensure that $T'$ is the image of $T$
under some embedding $j \in \Po(T, \SS^3; \xi)$. But the complement of $T'$ is
universally tight whereas $\xi_{|W}$ becomes overtwisted in a two-fold cover.
So $j$ is not in the component of the inclusion in $\Po(T, \SS^3; \xi)$.
\end{proof}

\begin{remark}
In order to get the weaker result that some solid torus $W$ satisfies the
conclusion of the above proposition, it is sufficient to observe that the
complement of $D$ contains an unknotted Legendrian knot $L$ entwining $D$ and
define $W$ as the complement of a standard neighborhood of $L$. In that case we
already control the dividing set of $\partial W$ by construction.
Note also that the knot $L$ is not isotopic to the canonical Legendrian unknot
$L_0$ since the complement of the later is universally tight. The
classification of Legendrian unknots in \cite{Eliashberg_Fraser} guarantees
that $L$ is stabilization of $L_0$. So in order to entwine $D$, one needs a
somewhat tortuous Legendrian unknot.
\end{remark}

Next we want to describe examples where we have explicit smooth isotopies among
surfaces which are all $\xi$-convex except for a finite number of times and
exhibit various behaviors for those isotopies. We also want to highlight
situations where persistent intersection phenomena occur and situations where a
contact isotopy exists in the ambiant manifold but not inside a smaller
manifold (where a smooth isotopy still exists). For all this we need the
following technical definition.

\begin{definition}
A discretized isotopy of embeddings of an oriented surface $S$ into a contact
3-manifold $(V, \xi)$ is an isotopy of embeddings $j \co S \times [0,1]
\to V$ such that, for some (unique) integer $n$:
\begin{itemize}
\item 
	the restriction of $j$ to $S \times [i/n, (i + 1)/n]$ is an embedding
	for each $i$ from $0$ to $n-1$,
\item
	all surfaces $j_t(S)$ are $\xi$-convex except when $t = i/n + 1/(2n)$ for some
	integer $i$ between $0$ and $n-1$.
\end{itemize}
Each embedding of $S \times [i/n, (i + 1)/n]$ is called a step of the
discretized isotopy. It is called a forward or backward step depending on
whether it is orientation preserving or reversing.
\end{definition}

Colin's idea described in our proof of Proposition~\ref{prop:tores_Paolo_Patrick}
combines with \cite[Lemme~15]{Giroux_transfo} to prove that any isotopy of
embeddings which starts and end at $\xi$-convex embeddings is homotopic
relative to its end-points to a discretized isotopy.

Any discretized isotopy $j$ defines a sequence of isotopy classes
of multi-curves $\Gamma_0, \dots, \Gamma_n$ such that the characteristic
foliation of $j_{i/n}(S)$ is divided by $j_{i/n}(\Gamma_i)$. 

Our examples below will use facts about $\SS^1$-invariant contact
structures on circle bundles which we now recall. Let $S$ be an oriented
surface with non-empty boundary and $V = S \times \SS^1$ seen as a circle
bundle over $S$. Let $\Gamma$ be a properly embedded multi-curve on $S$ such
that components of the complement of $S$ are labelled by plus or minus so that
adjacent components have different signs. Lutz proved in \cite{Lutz_cercles}
that there is a cooriented $\SS^1$-invariant positive contact structure $\xi$
on $V$ which is tangent to fibers exactly along $\Gamma \times \SS^1$ and
positively (resp. negatively) transverse to fibers over positive (resp.
negative) components of $S \setminus \Gamma$. One says that $\Gamma$ is the
dividing set of $\xi$ (it is indeed a dividing set for each surface 
$S \times \{\theta\}$).
Lutz also proved that two invariant contact structures which agree near
$\partial V$ and have the same dividing set are isotopic relative to
$\partial V$.
Let $\xi$ be such a contact structure. One can check that, for any properly
embedded curve $\gamma$ in $S$ which intersects the dividing set $\Gamma$
transversely (along a non-empty subset), the surface $\gamma \times \SS^1$ is
$\xi$-convex and divided by $(\gamma \cap \Gamma) \times \SS^1$.
Proposition~\ref{prop:decoupage} and Bennequin's theorem can be used to prove
that $\xi$ is tight if and only if $\Gamma$ has no homotopically trivial
component or $S$ is a disk and $\Gamma$ is connected, see \cite{Giroux_cercles}.
In addition two tight $\SS^1$-invariant contact structures on $V$ are isotopic
(relative to $\partial V$) if and only if their dividing sets are isotopic
(relative to $\partial S$). This is stated only for closed surfaces in
\cite{Giroux_cercles} but the proof is only easier if the boundary of $S$ is
not empty.

Recall from Section~\ref{ss:torus} that the contact structure $\xi_d$
on $\TT^3$ with coordinates $(x, y, z)$ is defined by:
\[
\xi_d = \ker\big( \cos(2d\pi z)dx - \sin(2d\pi z)dy\big)
\]
and they are pairwise non-isomorphic.

\begin{proposition}
\label{prop:isotopiesT3}
In $(\TT^3, \xi_d)$, let $T$ be the torus $\{8d\,z = \cos x \}$. Denote by $j_0$
the inclusion of $T$ into $\TT^3$ and by $j_1$ the embedding obtained by
restriction to $T$ of the rotation $(x, y, z) \mapsto (x, y, z + 1/d)$.
Those two embeddings are smoothly isotopic and:
\begin{itemize}
\item 
	$j_0$ and $j_1$ induce the same characteristic foliation on $T$
\item 
	$j_0$ is not isotopic to $j_1$ among $\xi_d$-convex embeddings,
\item
	there is a discretized isotopy from $j_0$ to $j_1$ with only forward steps
	changing the direction of dividing curves,
\item
	there is a discretized isotopy from $j_0$ to $j_1$ consisting of four forward steps
	which change the number of dividing curves without changing their direction.
\end{itemize}
\end{proposition}

\begin{proof}
Since the rotation map is a contactomorphism, $j_0$ and $j_1$ induce the same
characteristic foliation on $T$. Assume for contradiction that $j_0$ and $j_1$
are isotopic through $\xi_d$-convex surfaces. 
Proposition~\ref{prop:divided_foliations}c and
Lemma~\ref{lemma:fibration_PFxi} then imply that there is a contact isotopy
$\varphi$ such that $j_1 = \varphi_1 \circ j_0$. We lift this isotopy to 
$\TT^2 \times \RR$ which covers $\TT^3$ by 
$(x, y, s) \mapsto (x, y, s \mod 2\pi)$. We denote by $\varphi'$ the
lifted isotopy and by $T'$ some (fixed) lift of $T$.
We denote by $\tau_n$ the translation
$(x, y, s) \mapsto (x, y, s + n)$ and by $T_{[a, b]}$ the
compact manifold bounded by $\tau_a(T')$ and $\tau_b(T')$.
Because $T'$ is compact
and contact isotopies can be cut-off, we can assume that $\varphi'$ is compactly
supported. Then there is some $N$ such that $\varphi_1$ sends $T_{[-N, 0]}$ to
$T_{[-N, 1]}$. In particular those submanifolds are contactomorphic. This
contradicts the classification of tight contact structures on $\TT^3$ since
this contactomorphism could be used to build a contactomorphism from 
$(\TT^3, \xi_{N+1})$ to $(\TT^3, \xi_{N+2})$.

The existence of a discretized isotopy from $j_0$ to $j_1$ consisting of
forward steps changing the direction of dividing curves follows from repeated
uses of a small part of the classification of tight contact structures on
thickened tori: if $\xi$ is a tight contact structure on $\TT^2 \times [0, 1]$
such that $\TT^2 \times \{0\}$ and $\TT^2 \times \{1\}$ are $\xi$-convex with
two dividing curves $\gamma_0, \gamma_0'$ and $\gamma_1, \gamma_1'$
respectively such $\gamma_0$ intersects $\gamma_1$ transversely at one point
then $\xi$ is isotopic relative to the boundary to a contact structure $\xi'$
such that all tori $\TT^2 \times \{t\}$ are $\xi'$-convex except 
$\TT^2 \times \{1/2\}$.

In order to construct a discretized isotopy where the direction of dividing
curves is constant, we see $\xi_d$ as an $\SS^1$-invariant contact structure on
$\TT^3$ with $\SS^1$ action given by rotation in the $y$ direction. In order to
describe an $\SS^1$-equivariant isotopy of embeddings of $T$, it is enough to
give a isotopy of curves in $\TT^2$. Curves corresponding to are $\xi_d$-convex
tori are exactly those which are transverse to $\Gamma = \{ x \in (\pi/d) \ZZ\}$.
Figure~\ref{fig:discretizedT2} then finishes the proof.
\begin{figure}[ht]
	\centering
	\includegraphics{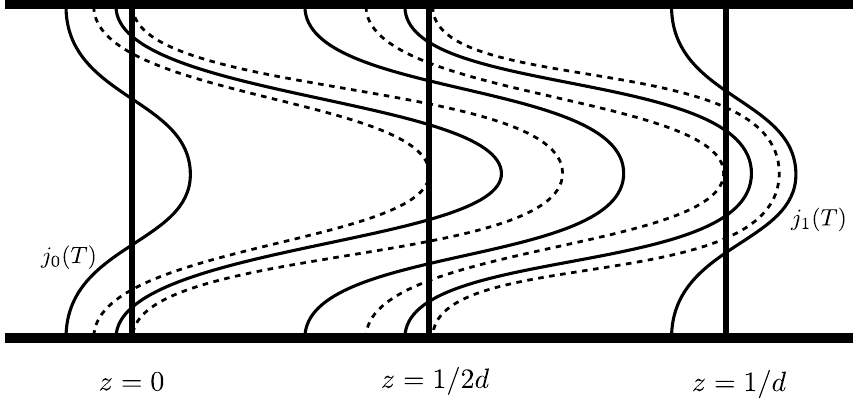}
	\caption{Discretized isotopy of curves lifting to tori in 
		$(\TT^3, \xi_d)$. Curves lifting to non-convex tori are dashed.}
	\label{fig:discretizedT2}
\end{figure}
\end{proof}

In our next example the discretized isotopy oscillates and there is persistent
intersection.

\begin{proposition}
\label{prop:isotopiesT3A}
Let $V$ be the torus bundle over $\SS^1$ with monodromy
$B = \left(\begin{smallmatrix} 5 & 1 \\ 4 & 1 \end{smallmatrix}\right)$,
ie.
\[
	V = (\TT^2 \times \RR)/\left( (Bx, t) \sim (x, t + 1)\right).
\]
Let $T$ be the image of $\TT^2 \times \{1/2\}$ in $V$ and let $j_0$ be the
inclusion map from $T$ to $V$.
There is a tight virtually overtwisted contact structure $\xi$ on $V$ and 
an embedding $j_1 \in \Po(T, V ; \xi)$ such that:
\begin{itemize}
\item 
$j_0$ is not isotopic to $j_1$ in $\Po(T, V ; \xi)$;
\item 
any $j \in \Po(T, V ; \xi)$ such that $j(T)$ is disjoint from $T$ is isotopic
to $j_0$ in $\Po(T, V ; \xi)$ (in particular $j_1(T)$ cannot be disjoined from
$T$ by contact isotopy);
\item
there is a discretized isotopy from $j_0$ to $j_1$ with one forward
step and one backward step, both modifying the direction of dividing curves. 
\end{itemize}
\end{proposition}

\begin{proof}
We will use the theory of normal forms for tight contact structures on $V$
established in \cite[Section~3]{Giroux_bif}	(see also \cite{Honda_I}). With the
notation used in \cite{Giroux_bif}, $V = T^3_A$ where 
$A = B^{-1} = \left(\begin{smallmatrix} 1 & -1 \\ -4 & 5
\end{smallmatrix}\right)$. Other choices of monodromies are possible, we only
want to explain one simple example.

We will describe the lifts of relevant contact structures on
$\TT^2 \times [0, 1] \subset \TT^2 \times \RR$. In \cite{Giroux_bif}, tight
contact structures on thickened tori are described using two types of building
blocks: rotation sequences and orbit flips that we will briefly review.

Recall that a foliation $\sigma$ on a torus $T$ is called a
suspension if there is a circle which transversely intersects all leaves. It
then has an asymptotic direction $d(\sigma)$ which is a line through the origin
in $H_1(T; \RR)$ spanned by limits of renormalized very long orbits of a
directing vector field.

We fix a contact structure $\xi$ on $T \times [0, 1]$ and set
$T_a = T \times \{a\}$. A interval 
$J \subset [0, 1]$ is called a rotation sequence for $\xi$ if all
characteristic foliations $\xi T_t$, $t \in J$ are suspensions.
We say that $J$ is minimally twisting if the directions $d(\xi T_t)$ 
do not sweep out the full projective line $P(H_1(T; \RR))$. Theorem 3.3 from
\cite{Giroux_bif} guaranties that two contact structures on
$T \times J$ which agree along the boundary and have $J$ as a minimally twisting 
rotation sequence are isotopic on $T \times J$ relative to boundary.

An interval $[a, b] \subset [0, 1]$ is an orbit flip sequence for $\xi$ with
homology class $d \in H_1(T; \ZZ)$ if:
\begin{itemize}
	\item $\xi T_a$ is a Morse-Smale suspension with two closed orbits whose
		homology classe is $d$ ;
	\item $\xi T_b$ is a Morse-Smale suspension with two closed orbits whose
		homology classe is $-d$ ;
	\item there is a multi-curve which divides all $\xi T_t$, $t \in J$.
\end{itemize}
The uniqueness lemma \cite[Lemma 2.7]{Giroux_bif} ensures that two contact
structures on $T \times [a, b]$ which agree along the boundary and admit 
$[a, b]$ as an orbit flip sequence are isotopic relative to boundary. There is
an explicit model in \cite[Section 1.F]{Giroux_bif} where all $\xi T_t$ are
suspensions except one which has two circles of singularities instead of
regular closed leaves.

Here we need two (isotopic) contact structures on $\TT^2 \times [0, 1]$. 
We fix a Morse-Smale suspension $\sigma_0$ on $\TT^2$ with two closed orbits
having homology class $(1, 0)$ and we denote by $\sigma_1$ the image of
$\sigma_0$ under $A$. We also fix a Morse-Smale suspension $\sigma_{1/2}$ 
with two closed orbits having homology class $(-1, 1)$.
Let $\xi$ be a contact structure on $\TT^2 \times [0, 1]$ such that
\begin{itemize}
	\item $\xi$ prints $\sigma_t$ on $\TT^2 \times \{t\}$ for 
		$t \in \{0, 1/2, 1\}$.
	\item $[0, 1]$ is a union of minimaly twisting rotation sequences and two
		orbit flip sequences with homology classes $(1,0)$ and $(-1, 1)$
		respectively.
\end{itemize}
Let $\xi'$ be a contact structure with the same properties except that orbits
flips homology classes are $(1, -1)$ and $(-1, 2)$.
Results from \cite[Section~3]{Giroux_bif} guaranty that $\xi$ and $\xi'$ 
induce isotopic tight contact structures on $V$. In addition the restriction of
$\xi$ to the complement of the image $T$ in $V$ of $\TT^2_{1/2}$ is 
universally tight whereas restriction of $\xi'$ is virtually overtwisted.
If $\varphi$ is a smooth isotopy such that $\xi' = \varphi_1^*\xi$ then the
inclusion $j_0 : T \inclusion V$ and $j_1 = \varphi_1 \circ j_0$ are not in the
same component of $\Po(T, V; \xi)$.

The second point of the proposition follows again from classification results.
Let $j$ be an embedding in $\Po(T, V ; \xi)$ such that $j(T)$ is disjoint from
$T$. The classification of incompressible tori in the complement of $T$
guaranties that $T$ and $j(T)$ bound a thickened torus $N$ in $V$. The
classification of tight contact structures on thickened tori in
\cite{Giroux_bif} or \cite{Honda_I} ensures that either $j$ is isotopic to 
$j_0$ in $\Po(T, V ; \xi)$ or there exist a contact embedding of  
\[
\left(T^2 \times [0, 1], \ker\left(\cos(\pi z)dx - 
\sin(\pi z)dy\right)\right), \qquad (x,y,z) \in T^2 \times [0,1]
\] 
into the interior of $N$. But the existence of such an embedding is ruled out
by the study of tight contact structure on $V$, specifically 
\cite[Proposition~1.8]{Giroux_bif}.

The announced discretized isotopy use the image of $\TT^2 \times \{0\}$ as a
intermediate surface and its existence if guarantied by the classification
result quoted in the proof of Proposition~\ref{prop:isotopiesT3}.
\end{proof}

Finally we describe an example on a manifold with boundary with the same
situation as above but things untangle inside a larger manifold.

\begin{proposition}
\label{prop:needs_space}
Let $V$ denote the manifold $\TT^2 \times [0, 1]$ and 
$V' = \TT^2 \times [0, 1/2]$.
There is a universally tight contact structure $\xi$ on $V$ and two
smoothly isotopic $\xi$-convex embeddings $j_0, j_1 : \TT^2 \to V'$ with
images $T_0$ and $T_1$ such that
\begin{itemize}
\item 
$j_0$ is isotopic to $j_1$ among $\xi$-convex embeddings in $V$
\item 
$j_0$ is not isotopic to $j_1$ among $\xi$-convex embeddings in $V'$
\item
$T_0$ cannot be disjoined from $T_1$ by an isotopy among $\xi$-convex surfaces in 
$V'$
\item
there is a discretized isotopy from $j_0$ to $j_1$ in $V'$ with one forward
step and one backward step, both modifying the direction of dividing curves. 
\item
there is a discretized isotopy from $j_0$ to $j_1$ in $V'$ with one forward
step and one backward step, both modifying the number of dividing curves. 
\end{itemize}
\end{proposition}

\begin{proof}
The construction is pictured in Figure~\ref{fig:needs_space}.
\begin{figure}[ht]
	\centering
	\includegraphics{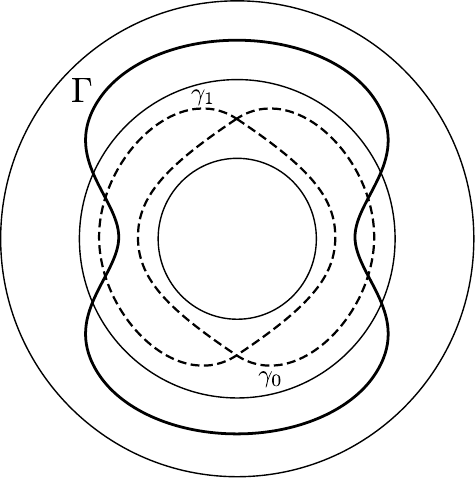}
	\caption{The example of Proposition~\ref{prop:needs_space}. The dividing
set $\Gamma$ is the thick curve, $\gamma_0$ and $\gamma_1$ are dashed.}
	\label{fig:needs_space}
\end{figure}
Let $S$ be the annulus $\{ 1 \leq |z| \leq 3\} \subset \CC$ and
$S' \subset S$ the subannulus $\{ 1 \leq |z| \leq 2\}$. We fix an
identification between $V$ and $S \times \SS^1$ which identify $V'$ with
$S' \times \SS^1$.
Let $\Gamma' = \Gamma'_1 \cup \Gamma'_2$ be a disjoint union of two
properly embedded arcs in $S'$ whose end points are on the circle 
$\{ |z| = 2\}$.
Let $\Gamma$ be a smooth homotopically essential circle in $S$ such that
$\Gamma \cap S' = \Gamma'$. Let $\xi$ be a $\SS^1$-invariant contact structure
on $V$ with dividing set $\Gamma$ and denote by $\xi'$ its restriction to $V'$.
Let $\gamma_0$ and $\gamma_1$ be homotopically essential circles in $S'$ such
that $\gamma_i$ intersects transversely $\Gamma'_i$ in two points and does not
intersect the other component of $\Gamma'$. The tori we want are 
$T_0 = \gamma_0 \times \SS^1$ and $T_1 = \gamma_1 \times \SS^1$, parametrized
by product maps.

There is an isotopy through $\xi$-convex surfaces in $V$ because $\gamma_0$ and
$\gamma_1$ are isotopic in $S$ through curves transverse to $\Gamma$.

Assume for contradiction that there is such an isotopy in $V'$. We can arrange
$\xi$ so that $j_0$ and $j_1$ induce the same characteristic foliation on
$\TT^2$ and, using Proposition~\ref{prop:divided_foliations}c and
Lemma~\ref{lemma:fibration_PFxi}, our isotopy through convex surfaces can then
be converted to a contact isotopy $\varphi$ relative to the boundary. 
We denote by $\partial_1 V'$ and $\partial_2 V'$ the connected components 
$\{ |z| = 1 \} \times \SS^1$ and $\{ |z| = 2 \} \times \SS^1$ of $\partial V'$.
Let $\psi_0$ and $\psi_1$ be smooth embeddings
of $V'$ into itself such that: 
\begin{itemize}
	\item Each $\psi_i$ is $\SS^1$-equivariant
	\item Each $\psi_i$ is the identity on $\partial_2 V'$
	\item $\psi_i(\partial_1 V') = T_i$.
\end{itemize}
The contact structures $\psi_0^*\xi$ and $\psi_1^*\xi$ on $V'$ are
$\SS^1$-invariant and the associated dividing sets $\Gamma_{\psi_0}$ and
$\Gamma_{\psi_1}$ are shown on Figure~\ref{fig:psi0_psi1}.
\begin{figure}[ht]
	\centering
	\includegraphics{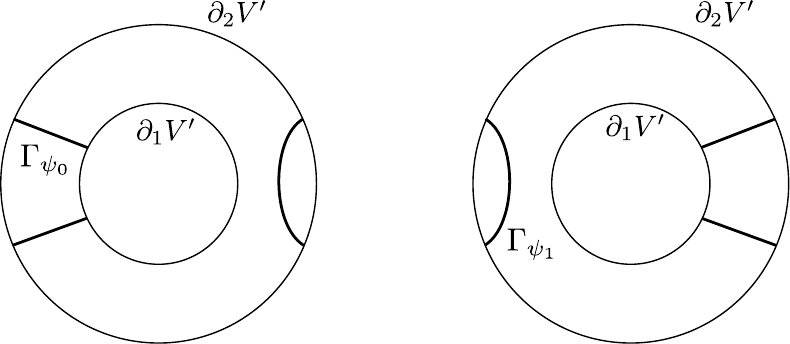}
	\caption{Dividing curves for the proof of Proposition~\ref{prop:needs_space}.
Thin curves are boundary components of $S'$ and thick curves are the components
of the dividing sets}
	\label{fig:psi0_psi1}
\end{figure}
The contactomorphism $\varphi_1$ then induces a contactomorphism between 
$(V', \psi_0^*\xi)$ and $(V', \psi_1^*\xi)$ which is the identity on
$\partial_2 V'$. However the classification of $\SS^1$-invariant contact
structures forbids the existence of this contactomorphism. In this case we can
argue directly as follows. We denote by $S''$ the annulus 
$\{ 2 \leq |z| \leq 3\}$. Let $\xi''$ be a contact structure on 
$S'' \times \SS^1$ which is $\SS^1$-invariant and tangent to $\SS^1$ along some
$\Gamma''$ such that $\Gamma_{\psi_0} \cup \Gamma''$ has a homotopically
trivial component but $\Gamma_{\psi_1} \cup \Gamma''$ has not. The contact
structure $\xi'' \cup \psi_0^*\xi$ on $V$ is overtwisted whether 
$\xi'' \cup \psi_1^*\xi$ is tight so we have a contradiction.

So there is no contact isotopy $\varphi$ in $V'$ such that 
$j_1 = \varphi_1 \circ j_0$.
Assume for contradiction that there is a contact isotopy $\varphi$ in $V'$ such
that $T_0' = \varphi_1(T_0)$ is disjoint from $T_1$. The classification of
incompressible surfaces in thickened tori ensures that $T_0' \cup T_1$ is the
boundary of a thickened torus in the interior of $V'$. After some smooth
deformation, we can assume that $T_0' = \{|z| = r_0\} \times \SS^1$ and 
$T_1 = \{|z| = r_1\} \times \SS^1$ and the contact structure is
$\SS^1$-invariant near $T_0'$ and $T_1$ (note that we don't know the sign of
$r_0 - r_1$). Those tori are both divided by
vertical curves $\{*\} \times \SS^1$ so the classification of universally tight
contact structures on tickened tori guaranties that, after some further isotopy
relative to $T_0'$ and $T_1$, the contact structure is $\SS^1$-invariant
everywhere (see \cite[Théorème~4.4]{Giroux_bif} or \cite{Honda_I}). The dividing
set in the annulus $\{ |z| \in [r_0, r_1]\}$ intersects each boundary component
in two points so it is either two boundary parallel arcs and some closed
components or two traversing arcs. The first possibility is ruled out by the
classification of $\SS^1$-invariant contact structures up to (non-necessarily
invariant) isotopy since the full dividing set on $S'$ would not be isotopic
to $\Gamma'$. The second possibility is ruled out because $T_0'$ and $T_1$
would then be isotopic among $\xi'$-convex surfaces, contradicting the previous
point.

The construction of discretized isotopies is completely analogous to what we
discussed for Proposition~\ref{prop:isotopiesT3}.
\end{proof}


\bib{Cerf_these}{article}{
  author={Cerf, Jean},
  title={Topologie de certains espaces de plongements},
  language={French},
  journal={Bull. Soc. Math. France},
  volume={89},
  date={1961},
  pages={227--380},
  issn={0037-9484},
}

\bib{Colin_spheres}{article}{
  author={Colin, Vincent},
  title={Chirurgies d'indice un et isotopies de sph\`eres dans les vari\'et\'es de contact tendues},
  language={French, with English and French summaries},
  journal={C. R. Acad. Sci. Paris S\'er. I Math.},
  volume={324},
  date={1997},
  number={6},
  pages={659--663},
  issn={0764-4442},
  doi={10.1016/S0764-4442(97)86985-6},
}

\bib{Colin_recollement}{article}{
  author={Colin, Vincent},
  title={Recollement de vari\'et\'es de contact tendues},
  language={French, with English and French summaries},
  journal={Bull. Soc. Math. France},
  volume={127},
  date={1999},
  number={1},
  pages={43--69},
  issn={0037-9484},
}

\bib{GeigesT3}{article}{
  author={Geiges, Hansjörg},
  author={Klukas, Mirko},
  title={The fundamental group of the space of contact structures on the 3-torus},
  journal={Math. Res. Lett.},
  volume={21},
  pages={1257--1262},
  year={2014},
}

\bib{Eliashberg_20ans}{article}{
  author={Eliashberg, Yakov},
  title={Contact $3$-manifolds twenty years since J. Martinet's work},
  language={English, with French summary},
  journal={Ann. Inst. Fourier (Grenoble)},
  volume={42},
  date={1992},
  number={1-2},
  pages={165--192},
  issn={0373-0956},
}

\bib{Fraser_P2}{article}{
  author={Fraser, Maia},
  title={Isotopies of real projective planes in tight contact $3$-manifolds.},
  note={In preparation},
}

\bib{Eliashberg_Fraser}{article}{
  author={Eliashberg, Yakov},
  author={Fraser, Maia},
  title={Classification of topologically trivial Legendrian knots},
  conference={ title={Geometry, topology, and dynamics}, address={Montreal, PQ}, date={1995}, },
  book={ series={CRM Proc. Lecture Notes}, volume={15}, publisher={Amer. Math. Soc., Providence, RI}, },
  date={1998},
  pages={17--51},
  review={\MR {1619122 (99f:57007)}},
}

\bib{Eliashberg_Hofer_Salamon}{article}{
  author={Eliashberg, Y.},
  author={Hofer, H.},
  author={Salamon, D.},
  title={Lagrangian intersections in contact geometry},
  journal={Geom. Funct. Anal.},
  volume={5},
  date={1995},
  number={2},
  pages={244--269},
  issn={1016-443X},
  doi={10.1007/BF01895668},
}

\bib{Paolo_thesis}{article}{
  author={Ghiggini, Paolo},
  title={Tight contact structures on Seifert manifolds over $T^2$ with one singular fibre},
  journal={Algebr. Geom. Topol.},
  volume={5},
  date={2005},
  pages={785--833},
  issn={1472-2747},
  doi={10.2140/agt.2005.5.785},
}

\bib{Paolo_T3}{article}{
  author={Ghiggini, Paolo},
  title={Linear Legendrian curves in $T^3$},
  journal={Math. Proc. Cambridge Philos. Soc.},
  volume={140},
  date={2006},
  number={3},
  pages={451--473},
  issn={0305-0041},
  review={\MR {2225643 (2007a:57031)}},
  doi={10.1017/S0305004105009151},
}

\bib{Giroux_thesis}{article}{
  author={Giroux, Emmanuel},
  title={Convexit\'e en topologie de contact},
  language={French},
  journal={Comment. Math. Helv.},
  volume={66},
  date={1991},
  number={4},
  pages={637--677},
  issn={0010-2571},
  doi={10.1007/BF02566670},
}

\bib{Giroux_tordue}{article}{
  author={Giroux, Emmanuel},
  title={Une structure de contact, m\^eme tendue, est plus ou moins tordue},
  language={French, with English summary},
  journal={Ann. Sci. \'Ecole Norm. Sup. (4)},
  volume={27},
  date={1994},
  number={6},
  pages={697--705},
  issn={0012-9593},
}

\bib{Giroux_bif}{article}{
  author={Giroux, Emmanuel},
  title={Structures de contact en dimension trois et bifurcations des feuilletages de surfaces},
  language={French},
  journal={Invent. Math.},
  volume={141},
  date={2000},
  number={3},
  pages={615--689},
  issn={0020-9910},
  doi={10.1007/s002220000082},
}

\bib{Giroux_cercles}{article}{
  author={Giroux, Emmanuel},
  title={Structures de contact sur les vari\'et\'es fibr\'ees en cercles au-dessus d'une surface},
  language={French, with French summary},
  journal={Comment. Math. Helv.},
  volume={76},
  date={2001},
  number={2},
  pages={218--262},
  issn={0010-2571},
  doi={10.1007/PL00000378},
}

\bib{Giroux_transfo}{article}{
  author={Giroux, Emmanuel},
  title={Sur les transformations de contact au-dessus des surfaces},
  language={French},
  conference={ title={Essays on geometry and related topics, Vol. 1, 2}, },
  book={ series={Monogr. Enseign. Math.}, volume={38}, publisher={Enseignement Math.}, place={Geneva}, },
  date={2001},
  pages={329--350},
}

\bib{Hatcher_large}{article}{
  author={Hatcher, Allen},
  title={Homeomorphisms of sufficiently large $P^{2}$-irreducible $3$-manifolds},
  journal={Topology},
  volume={15},
  date={1976},
  number={4},
  pages={343--347},
  issn={0040-9383},
}

\bib{Honda_I}{article}{
  author={Honda, Ko},
  title={On the classification of tight contact structures. I},
  journal={Geom. Topol.},
  volume={4},
  date={2000},
  pages={309--368},
  issn={1465-3060},
  doi={10.2140/gt.2000.4.309},
}

\bib{Laudenbach_dim3}{book}{
  author={Laudenbach, Fran{\c {c}}ois},
  title={Topologie de la dimension trois: homotopie et isotopie},
  language={French},
  note={With an English summary and table of contents; Ast\'erisque, No. 12},
  publisher={Soci\'et\'e Math\'ematique de France},
  place={Paris},
  date={1974},
  pages={i+152},
}

\bib{Lutz_cercles}{article}{
  author={Lutz, Robert},
  title={Structures de contact sur les fibr\'es principaux en cercles de dimension trois},
  language={French, with English summary},
  journal={Ann. Inst. Fourier (Grenoble)},
  volume={27},
  date={1977},
  number={3},
  pages={ix, 1--15},
  issn={0373-0956},
  review={\MR {0478180 (57 \#17668)}},
}

\bib{Lutz_pivot}{article}{
  author={Lutz, Robert},
  title={Structures de contact et syst\`emes de Pfaff \`a pivot},
  language={French},
  conference={ title={Third Schnepfenried geometry conference, Vol. 1 (Schnepfenried, 1982)}, },
  book={ series={Ast\'erisque}, volume={107}, publisher={Soc. Math. France}, place={Paris}, },
  date={1983},
  pages={175--187},
}

\bib{Patrick_gcs}{article}{
  author={Massot, Patrick},
  title={Geodesible contact structures on 3-manifolds},
  journal={Geom. Topol.},
  volume={12},
  date={2008},
  number={3},
  pages={1729--1776},
  issn={1465-3060},
  doi={10.2140/gt.2008.12.1729},
}

\bib{Palais}{article}{
  author={Palais, Richard S.},
  title={Local triviality of the restriction map for embeddings},
  journal={Comment. Math. Helv.},
  volume={34},
  date={1960},
  pages={305--312},
  issn={0010-2571},
}

\bib{Waldhausen_klasse}{article}{
  author={Waldhausen, Friedhelm},
  title={Eine Klasse von $3$-dimensionalen Mannigfaltigkeiten. I, II},
  language={German},
  journal={Invent. Math. 3 (1967), 308--333; ibid.},
  volume={4},
  date={1967},
  pages={87--117},
  issn={0020-9910},
}



\begin{bibdiv}
\begin{biblist}
\bibselect{cmcg}
\end{biblist}
\end{bibdiv}
\end{document}